\documentclass[twocolumn,twoside,10pt]{siggraph}	

\usepackage{amssymb,amsfonts,amsmath}
\usepackage{amscd}
\usepackage{url}
\usepackage{hyperref}

\newtheorem{remark}{Remark}

\newcommand{\0}{{\bm 0}}

\newcommand{\A}{\bm{A}}

\newcommand{\B}{\bm{B}}

\newcommand{\Bmatr}[1]{\begin{bmatrix}\displaystyle #1\end{bmatrix}}

\newcommand{\D}{\bm{D}}

\newcommand{\I}{\bm{I}}

\newcommand{\M}{\bm{M}}

\renewcommand{\P}{\bm{P}}
\newcommand{\Pm}{\bm{P}}

\newcommand{\Realpart}[1]{\mathrm{Re}({#1})}

\newcommand{\an}[1]{\begin{align}#1\end{align}}
\newcommand{\ab}[1]{\begin{align*}#1\end{align*}}

\newcommand{\desclist}[1]{\begin{description}#1\end{description}}

\newcommand{\bm}[1]{{\bf #1}}

\newcommand{\case}[1]{\left\{ \begin{array}{ll} #1 \end{array}\right.}  

\newcommand{\df}[2]{\displaystyle \frac{\mbox{\rm d} #1}{\mbox{\rm d} #2}}
\newcommand{\dfinline}[2]{\mbox{\rm d} #1 / \mbox{\rm d} #2}

\newcommand{\enumlist}[1]{\begin{enumerate}#1\end{enumerate}}

\newcommand{\eqdef}{:=}

\newcommand{\ev}{\bm{e}}
\newcommand{\evt}{{\,\ev^\tp}}

\newcommand{\revision}[1]{#1}

\newcommand{\goesto}{\rightarrow}

\newcommand{\hide}[1]{}	

\newcommand{\ov}{\overline}

\newcommand{\pf}[2]{\displaystyle \frac{\partial #1}{\partial #2}}

\newcommand{\pfinline}[2]{\partial #1 / \partial #2}

\newcommand{\spb}{{\rm spb}}

\newcommand{\suchthat}{\colon}

\newcommand{\tp}{\top}	

\newcommand{\tr}{\!^\top}
\newcommand{\type}{\omega}

\newcommand{\uv}{\bm{u}}
\newcommand{\uvt}{\uv^\tp}

\newcommand{\vv}{\bm{v}}

\newcommand{\x}{\bm{x}}

\newfont{\gilfont}{cmsy10 scaled\magstep0}
\newcommand{\Reals}{\mathbb{R}} 
\newcommand{\Complex}{\mathbb{C}} 

\newtheorem{Theorem}{Theorem}

\newtheorem{Definition}{Definition}
\newtheorem{Lemma}{Lemma}

\newtheorem{Theorem:}[Theorem]{Theorem:}
\newtheorem{Conjecture:}[Theorem]{Conjecture:}
\newtheorem{Corollary:}[Theorem]{Corollary:}
\newtheorem{Definition:}[Theorem]{Definition:}
\newtheorem{Lemma:}[Theorem]{Lemma:}
\newtheorem{Paradox:}[Theorem]{Paradox:}
\newtheorem{Principle:}[Theorem]{Principle:}
\newtheorem{Proposition:}[Theorem]{Proposition:}
\newtheorem{Recursion:}[Theorem]{Recursion:}
\newtheorem{Result:}[Theorem]{Result:}

\newtheorem{Theorem-InOrder}[Theorem]{Theorem}
\newtheorem{Conjecture-InOrder}[Theorem]{Conjecture}
\newtheorem{Corollary-InOrder}[Theorem]{Corollary}
\newtheorem{Definition-InOrder}[Theorem]{Definition}
\newtheorem{Lemma-InOrder}[Theorem]{Lemma}
\newtheorem{Paradox-InOrder}[Theorem]{Paradox}
\newtheorem{Principle-InOrder}[Theorem]{Principle}
\newtheorem{Proposition-InOrder}[Theorem]{Proposition}
\newtheorem{Recursion-InOrder}[Theorem]{Recursion}
\newtheorem{Result-InOrder}[Theorem]{Result}
\newtheorem{Counterexample-InOrder} [Theorem]{Counterexample}
\renewcommand{\spb}{s}

\newcommand{\citet}{\cite}
\newcommand{\citep}{\cite}
\newcommand{\qedsym}{{ $\square$}}
\newcommand{\qedhere}{\tag*{\qedsym}}
\newcommand{\dropcap}[1]{#1}
\title {Resolvent Positive Linear Operators \\Exhibit the Reduction Phenomenon}
\author{Lee Altenberg \\ \url{altenber@hawaii.edu} 
}
\newcommand{\abstractext}{The spectral bound, $\spb(\alpha A + \beta V)$, of a combination of a resolvent positive linear operator $A$ and an operator of multiplication $V$, was shown by Kato to be convex in $\beta \in \Reals$.  \revision{This is shown here to imply, through an elementary lemma, that $\spb(\alpha A + \beta V)$ is also convex in $\alpha > 0$, and notably, $\partial \, \spb(\alpha A + \beta V) / \partial \alpha \leq \spb(A )$ when it exists.  Diffusions typically have $\spb(A) \leq 0$, so that for diffusions with spatially heterogeneous growth or decay rates, \emph{greater mixing reduces growth}.  Models of the evolution of dispersal in particular have found this result when $A$ is a Laplacian or second-order elliptic operator, or a nonlocal diffusion operator, implying selection for reduced dispersal.  These cases are shown here to be part of a single, broadly general, `reduction' phenomenon.}}

\begin{document}
\maketitle
\abstract{
\abstractext\footnote{Dedicated to Sir John F. C. Kingman on the fiftieth anniversary his theorem on the `superconvexity' of the spectral radius \citep{Kingman:1961:Convexity}, which is at the root of the results presented here.}
}

\dropcap{T}he main result to be shown here is that the growth bound, $\type(m A + V)$, of a positive semigroup generated by $m A + V$ changes with positive scalar $m$ at a rate less than or equal to $\type(A)$, where $A$ is also a generator, and $V$ is an operator of multiplication.  Movement of a reactant in a heterogeneous environment is often of this form, where $V$ represents the local growth or decay rate, and $m$ represents the rate of mixing.  Lossless mixing means $\type(A) = 0$, while lossy mixing means $\type(A) < 0$, so this result implies that greater mixing reduces the reactant's asymptotic growth rate, or increases its asymptotic decay rate.  This is a familiar result when $A$ is a diffusion operator, so what is new here is the generality \revision{shown for this phenomenon}.  At the root of this result is a theorem by Kingman on the `superconvexity' of the spectral radius of nonnegative matrices \citep{Kingman:1961:Convexity}.  The logical route progresses from Kingman through Cohen \citet{Cohen:1981:Convexity} to Kato \citet{Kato:1982:Superconvexity}.  The historical route begins in population genetics.

{I}n early theoretical work to understand the evolution of genetic systems, Feldman and colleagues kept finding a common result from each model they examined \revision{\cite{Feldman:1972,Balkau:and:Feldman:1973,Karlin:and:McGregor:1974,Feldman:and:Krakauer:1976,Teague:1976,Teague:1977,Feldman:Christiansen:and:Brooks:1980,Feldman:and:Liberman:1986}} --- be they models for the evolution of recombination, or of mutation, or of dispersal.  Evolution favored reduced levels of these processes in populations near equilibrium under constant environments, and this result was called the \emph{Reduction Principle} \citep{Feldman:Christiansen:and:Brooks:1980}.

These results were found for finite-dimensional models.  But the same reduction result has also been found in models for the evolution of unconditional dispersal in continuous space, in which matrices are replaced by linear operators.  This raises the questions of whether this common result, discovered in such a diversity of models, reflects a single mathematical  phenomenon.  Here, the question is answered affirmatively.  

The mathematical underpinnings of the reduction principle for finite-dimensional models were discovered by Sam Karlin \cite{Karlin:1976,Karlin:1982} (although he did not realize it, and he had earlier proposed an alternate to the reduction principle --- the \emph{mean fitness principle} \cite{Karlin:and:McGregor:1972:Modifier}, which was found to have counterexamples  \cite{Karlin:and:Carmelli:1975}).  Karlin wanted to understand the effect of population subdivision on the maintenance of genetic variation.  Genetic variation is preserved if an allele has \revision{a} positive growth rate when it is rare, protecting it from extinction.  The dynamics of a rare allele are approximately linear, and of the form
\an{\label{eq:Karlin}
\x(t{+}1) = [ (1{-}m) \I + m \P] \D \, \x(t)
}
where $\x(t)$ is a vector of the rare allele's frequency among different population subdivisions, $m$ is the rate of dispersal between subdivisions,  $\Pm$ is the stochastic matrix representing the pattern of dispersal, and $\D$ is a diagonal matrix of the growth rates of the allele in each subdivision.  The allele is protected from extinction if its asymptotic growth rate when rare is greater than 1.  This asymptotic growth rate is the spectral radius, 
\an{
\rho(\A) \eqdef  \max \{| \lambda | \suchthat \lambda \in \sigma(\A)\},
}
where $\sigma(\A)$ is the set of eigenvalues of matrix $\A$.

Karlin discovered that for $\M(m) \eqdef [ (1{-}m) \I + m \P]$, the spectral radius, $\rho(\M(m) \D)$, is a decreasing function of the dispersal rate $m$, for arbitrary strongly-connected dispersal pattern:
\ 

\begin{Theorem}[Karlin Theorem 5.2, { \citep[pp. 194--196]{Karlin:1982}}]\label{Theorem:Karlin5.2}
Let $\P$ be an arbitrary non-negative irreducible stochastic matrix.  Consider the family of matrices
 \[
 \M{(\alpha)} = (1-\alpha) \I + \alpha \P.
 \]
Then for any diagonal matrix $\D$ with positive terms on the diagonal, the spectral radius
\[
\rho(\alpha) = \rho( \M{(\alpha)} \D )
\]
is decreasing as $\alpha$ increases (strictly provided $\D \neq d \I$).
\end{Theorem}
\ 

Theorem 5.2 means that greater mixing between subdivisions produces lower $\rho(\M(m) \D)$, and if it crosses below $1$, the allele will go extinct.  While this theorem was motivated by the issue of genetic diversity in a subdivided population, the generality of its form applies to any situation where differential growth is combined with mixing.  $\D$ could just as well represent the investment returns on different assets and $\P$ a pattern of portfolio rebalancing.  Or $\D$ could represent the decay rates of reactant in different parts of a reactor, and $\P$ a pattern of stirring within the reactor.  In a very general interpretation, Theorem 5.2 means that \revision{\emph{greater mixing reduces growth and hastens decay}}.  

If the dispersal rate $m$ is not an extrinsic parameter, but is a variable which is itself controlled by a gene, then a gene which decreases $m$ will have a growth advantage over its competitor alleles.  The action of such modifier genes produces a process that will reduce the rates of dispersal in a population.  Therefore, Theorem 5.2 also means that \emph{differential growth selects for reduced mixing}.

In the evolutionary context, the generality of the mixing pattern $\P$ in Karlin's Theorem 5.2 makes it applicable to other kinds of `mixing' besides dispersal.  The pattern matrix $\P$ can just as well refer to the pattern of mutations between genotypes, and then $m$ refers to the mutation rate.  Or $\P$ can represent the pattern of transmission when two loci recombine, and then $m$ represents the recombination rate.  The early models for the evolution of recombination and mutation that exhibited the reduction principle in fact had the same form as \eqref{eq:Karlin} for the dynamics of a rare modifier allele.  Once this was recognized \citet{Altenberg:1984:book,Altenberg:and:Feldman:1987,Altenberg:2009:Linear}, it was clear that Karlin's theorem explained the repeated appearance of the reduction result in the different contexts, and generalized the result to a whole class of genetic transmission patterns beyond the special cases that had been analyzed.

The dynamics of movement in space have been long modeled by infinite-dimensional models, where space is continuous and the concentrations of a quantity at each point are represented as a function.  The dynamics of change in the concentration are modeled as diffusions, where the Laplacian or elliptic differential operator or nonlocal integral operator takes the place of the matrix $\P$ in the finite-dimensional case.  When the substance grows or decays at rates that are a function of its location, the system is often referred to as a reaction-diffusion.  In reaction-diffusion models for the evolution of dispersal, the reduction principle again makes its appearance \cite{Hastings:1983}\cite[Lemma 5.2]{Hutson:Lopez-Gomez:Mischaikow:and:Vickers:1995:Limit} \cite[Lemma 2.1]{Dockery:Hutson:Mischaikow:and:Pernarowski:1998:Slow}\revision{\cite{Cantrell:Cosner:and:Lou:2010:Evolution}}.  In nonlocal diffusion models, again the reduction principle appears \revision{\cite{Hutson:Martinez:Mischaikow:and:Vickers:2003:Evolution}}.  This points to the possibility of an underlying mathematical unity.

Here,  a broad characterization of this `reduction phenomenon' is established by generalizing Karlin's theorem to linear operators.  The reduction results previously found for various linear operators are, therefore, shown to be special cases of a general phenomenon.  

This result is actually implicit in Kato's generalization \citet{Kato:1982:Superconvexity} of Cohen's theorem \citep{Cohen:1981:Convexity} on the convexity of the spectral bound of essentially nonnegative matrices with respect to the diagonal elements of the matrix.  It is \revision{educed} from Kato's theorem here by means of an elementary `dual convexity' lemma. 

Kato's goal in \citet{Kato:1982:Superconvexity} was to generalize, from matrices to linear operators, Cohen's  convexity result \citet{Cohen:1981:Convexity}\revision{:}
\ \\
\begin{Theorem}[Cohen {\citet{Cohen:1981:Convexity}}]\label{Theorem:Cohen}
Let $\D$ be diagonal real $n \times n$ matrix.  Let $\A$ be an essentially nonnegative $n \times n$ matrix. 

Then $\spb(\A {+} \D)$ is a convex function of $\D$.
\end{Theorem}
\ \\
Here, \revision{$\spb(\A {+} \D)$ is the spectral bound --- the largest real part of any eigenvalue of $\A{+}\D$.    A synonym for the spectral bound used in the matrix literature is the \emph{spectral abscissa} \citep{Lozinskiy:1969:Estimate,Deutsch:and:Neumann:1984,Deutsch:and:Neumann:1985}.  When the spectral bound is an eigenvalue, it is also referred to as the \emph{principal eigenvalue} \revision{\cite{Keilson:1964:Review}}, \emph{dominant eigenvalue} \citep{Horn:and:Johnson:1985}, \emph{dominant root} \citep{Gantmacher:1959:Applications}, \emph {Perron-Frobenius eigenvalue} \citep{Seneta:1981}, or \emph{Perron root} \citep{Bellman:1955:Iterative}.}  `Essentially nonnegative' means that the off-diagonal elements are nonnegative.  Synonyms include `{quasi-positive}' \citep{Hadeler:and:Thieme:2008:Monotone}, `Metzler', `Metzler-Leontief',  `ML' \citep{Seneta:1981}, and `cooperative' \citep{Birindelli:Mitidieri:and:Sweers:1999}:

Cohen's proof relied upon the following theorem of Kingman:  
\ \\
\begin{Theorem}[Kingman {\citet{Kingman:1961:Convexity}}]\label{Theorem:Kingman}
Let $\A$ be an $n \times n$ matrix whose elements, $A_{ij}(\theta)$, are non-negative functions of the real variable $\theta$, such that they are `superconvex', i.e. for each $i, j$, either $\log A_{ij}(\theta)$ is convex \revision{in $\theta$}, or $A_{ij}(\theta)=0$ for all $\theta$.

Then the spectral radius of $\A$ is also superconvex in $\theta$.
\end{Theorem}
\ 

Kato generalized Cohen's result to linear operators by first generalizing Kingman's theorem.  Before presenting Kato's theorem, some terminology needs to be introduced:  
\desclist{
\item[$X$] represents an ordered Banach space or its complexification.
\item[$X_+$] represents the proper, closed, positive cone of $X$, assumed to be generating and normal (see \cite{Kato:1982:Superconvexity}).
\item[$B(X)$] represents the set of all bounded linear operators $A \suchthat X \mapsto X$. 
\item[$A$] is a \emph{positive operator} if $A X_+ \subset X_+$.
\item[{\it The resolvent of $A$}] is $R(\xi, A ) \eqdef (\xi - A )^{-1}$, the operator inverse of $\xi - A $, $\xi \in \Complex$.  
\revision{\item[{\it The resolvent set} $\varrho(A) \subset \revision{\Complex}$] are those values of $\xi$ for which $\xi - A$ is invertible.}
\revision{\item[{\it The spectrum}] of $A \in B(X)$, $\sigma(A)$, is the complement of the resolvent set, $\varrho(A)$.}
\item[{\rm The} {\it spectral bound}] of closed linear operator $A$, not necessarily bounded, is 
\ab{
\spb(A ) \eqdef 
\case{
\sup \{ \Realpart \lambda\suchthat \lambda \in \sigma(A )\} & \text{if } \sigma(A ) \neq \emptyset \\
- \infty &\text{if }  \sigma(A ) = \emptyset.
}	
}
\item[{\it The type}] (growth bound) of an infinitesimal generator, $A$, of a strongly continuous ($C_0$) semigroup, $\{e^{t A }\suchthat t > 0\}$, is
\ab{
\type(A ) \eqdef 
\lim_{t \goesto \infty} \frac{1}{t}  \log \| e^{t A } \| .
}
Generally, $-\infty \leq \spb(A ) \leq \type(A ) < + \infty$, but conditions for $\spb(A ) = \type(A )$ or $\spb(A ) < \type(A )$ are part of a more involved theory for the asymptotic growth of semigroups (see \cite{Neerven:1996:Asymptotic}). 
}	

\begin{Definition}
$A$ is {\it resolvent positive} if there is $\xi_0$ such that $(\xi_0, \infty)  \subset \varrho(A)$ and $R(\xi, A )$ is positive for all $\xi > \xi_0$ \cite{Arendt:1987:Resolvent}.  
\end{Definition}

The relationship of the resolvent positive property to other familiar operator properties includes the following \revision{list of key} results:
\enumlist{
\item If $A$ generates a $C_0$-semigroup $T_t$, then $T_t$ is positive for all $t \geq 0$ if and only if $A$ is resolvent positive \citet[p. 188]{Arendt:Batty:Hieber:and:Neubrander:2011}.
\item If $A$ is a resolvent positive operator defined densely on $X = C(S)$, the \revision{Banach} space of continuous complex-valued functions on compact space $S$, then $A$ generates a positive $C_0$-semigroup \citet[Theorem 3.11.9] {Arendt:Batty:Hieber:and:Neubrander:2011}.
\item If $A$ is resolvent positive and its domain, $D(A) \subset X$, is dense in $X$, then for every $f \in D(A^2)$, there exists a unique solution, $u(t) \in D(A)$ for all $t \geq 0$, $u \in C^1([0, \infty), X)$, to the Cauchy problem \cite[Theorem 7.1]{Arendt:1987:Resolvent}
\ab{
\pf{u}{t} &= A u(t) \qquad (t \geq 0), \qquad u(0) = f.
}
\item  If $A$ is resolvent positive then: $\spb(A) < + \infty$; if $\sigma(A )$ is nonempty, i.e. $-\infty < \spb(A)$, then $\spb(A) \in \sigma(A)$; if $\xi \in \Reals \cap \varrho(A)$ yields $R(\xi,A) \geq 0$ then $\xi > \spb(A)$   \cite{Kato:1982:Superconvexity} \citet[Proposition 3.11.2]{Arendt:Batty:Hieber:and:Neubrander:2011}.
%
\item Differential operators higher than second order are never resolvent positive \cite[Corollary 2.3]{Arendt:Batty:and:Robinson:1990}\cite{Ulm:1999:Interval}.
\item Particular cases of resolvent positive operators include 
\enumlist{
\item second-order elliptic operators
 \ab{
 A = \sum_{j,k=1}^n a_{jk}(x) \frac{\partial^2}{\partial x_j \partial x_k} + \sum_{j=1}^n b_j(x) \pf{}{x_j} + c(x),
}
where the matrix $\Bmatr{a_{jk}(x)}_{j,k=1}^n$ is symmetric and positive-definite for each $x$, and appropriate regularity conditions hold for the domain and coefficients (e.g. \cite{Donsker:and:Varadhan:1976},\cite{Kato:1982:Superconvexity},\cite{Berestycki:Nirenberg:and:Varadhan:1994:Principal}).  
\item Linear integral operators $A $ on \revision{$X=C(\ov{\Omega})$} defined by
\ab{
(A f)(x) \eqdef \int_\Omega K(x,y) \, f(y)\,  dy + b(x) \, f(x),
}
where
$K \in C(\ov{\Omega} \times \ov{\Omega}, \Reals^+)$, $\Omega \subset \Reals^n$ is bounded, and $K(x, y) > 0$, $b(x)$ are measurable functions for $x,y \in \ov{\Omega}$ \citet{Hutson:Martinez:Mischaikow:and:Vickers:2003:Evolution,Grinfeld:Hines:Hutson:Mischaikow:and:Vickers:2005:Non-local,Bates:and:Zhao:2007:Existence}.  A resolvent positive combination of integral and differential operator is analyzed in \cite{Chabi:and:Latrach:2002:Singular}.
}	
}	

Kato's generalization of Cohen's theorem is as follows.  
\ \\
\begin{Theorem}[Generalized Cohen's theorem \citep{Kato:1982:Superconvexity}]\label{Theorem:Kato}
\ 

Consider $X= C(S)$ (continuous functions on a compact Hausdorff space $S$) or $X=L ^p(S)$, $l \leq p < +\infty$, on a measure space $S$, or more generally, let $X$ be the intersection of two $L ^p$-spaces with different $p$'s and different weight functions.  Let $A \suchthat X \mapsto X$ be a linear operator which is resolvent positive.  Let $V$ be an operator of multiplication on $X$ represented by a real-valued function $v$, where $v \in C(S)$ for $X=C(S)$, or $v \in L ^\infty(S)$ for the other cases. 

Then $\spb(A +V)$ is a convex function of $V$. If in particular $A$ is a generator, then both $\spb (A +V)$ and $\type(A +V)$ are convex in $V$.
\end{Theorem}
\section*{Results} 
\begin{Theorem} [Generalized Karlin's theorem]\label{Theorem:Main} \ \\
Let $A$ be a resolvent positive linear operator, and $V$ be an operator of multiplication, under the same assumptions as Theorem \ref{Theorem:Kato}.  

Then for $m>0$, 
\enumlist{
\item $\spb(m \, A + V)$ is convex in m;  \label{item:Theorem:Main:1}
\item For each $m>0$, either \label{item:Theorem:Main:2}
\enumlist{
\item $\spb((m+d) A + V) < \spb(m \, A + V) + d \ \spb(A ) \ \ \revision{\forall} \ d>0$, or 
\item $\spb((m+d) A + V) = \spb(m \, A + V) + d \ \spb(A ) \ \ \revision{\forall} \ d > 0$;
} 
\item In particular, when $\spb(A ) = 0$ then $\spb(m \, A + V)$ is non-increasing in $m$ (the `reduction phenomenon'), and when $\spb(A ) < 0$ then $\spb(m \, A + V)$ is strictly decreasing in $m$; \label{item:Theorem:Main:3}
\item Whenever $\df{}{m} \spb(m \, A + V)$ exists, then \label{item:Theorem:Main:4}
\an{
\df{}{m} \spb(m \, A + V) \leq \spb(A ). \label{eq:Main}
}
}	
If $A$ is a generator of a $C_0$-semigroup, then the above relations on $\spb(m \, A +V)$ also apply to the type $\type(m \, A +V)$.
\end{Theorem}
%
\ \\
\begin{proof}
We consider the general form
\an{\label{eq:GeneralForm}
\phi(\alpha, \beta) \eqdef \spb(\alpha A + \beta V) \text{ or } \type(\alpha A + \beta V)
}
where $\alpha >0, \beta \in \Reals$.  Kato \cite{Kato:1982:Superconvexity} explicitly shows that $\phi(1, \beta)$ is convex in $\beta$ (which he points out is equivalent to varying $V$).  Lemma \ref{Lemma:Main} (to follow) shows that this implies the properties asserted above regarding the effects of varying $m$ on $\spb(m \, A + V) = \phi(m, 1)$.  \qedsym
\end{proof}
\ \\
\begin{Lemma}[Dual Convexity]\label{Lemma:Main}
Let $f: \Reals \times \Reals \mapsto \Reals$, be jointly continuous.  For $x > 0$ and $y \geq 0$, let $f$ have the following properties:
\an{
f(\alpha x, \alpha y) &= \alpha f(x,y), \text{ for } \alpha > 0,  \label{eq:rescaler}
}
and
\an{
f(x,y) & \text{ is convex in }y. \label{eq:convexity}
}
Then for $x > 0$:
\begin{enumerate}
\item $f(x,y)$ is convex in $x$, for $y > 0$; \label{Lemma:xConvex}
\item For each $x > 0$, either \label{Lemma:Main:2}
\enumlist{
\item $f(x+d,1) < \ f(x,1) + d \, f(1,0)$ $\forall \ d > 0$; or\label{Lemma:x+d<}
\item $f(x+d,1) = \ f(x,1) + d \, f(1,0)$ $\forall \ d > 0$; \label{Lemma:x+d=}
}
\item When it exists, $\pf{}{x}f(x, 1) \leq f(1,0)$.  \label{Lemma:Derivative}
\end{enumerate}	
If, in the above, $f(x,y)$ is strictly convex in $y$, then $f(x,y)$ is strictly convex in $x$, and $f(x+d,1) < \ f(x,1) + d \, f(1,0)$.

The results is unchanged if the inequalities on $y$ are reversed.  
\end{Lemma}
%
\begin{proof}
{\flushleft \ref{Lemma:xConvex}.  \emph{$f(x,y)$ is convex in $x$, for $y > 0$}.}

The relation $f(\alpha x, \alpha y) = \alpha f(x,y)$ allows a set of rescalings that transform convexity in $y$ into convexity in $x$.  It is perhaps worth noting that this relation is actually a homomorphism, which can be put into a more familiar form by defining product $x \star y \eqdef f(x,y)$, and function $\psi(x) \eqdef \alpha  x$, which gives $\psi(x) \star \psi(y) = \psi (x \star y)$.  

For the following derivations, the constraints are $y \neq 0$, 
$y_1, y_2 \neq 0$ have the same sign as $y$, and
$0 < m < 1$.

These restrictions are made so that $\{y, y_1, y_2, m, 1{-}m$, $(1{-}m)y_1 + m y_2 \}$ are nonzero and all have the same sign, so that division with them is defined, and \revision{their ratios do not change sign} when the sign of $y$ is reversed.

Convexity of $f$ in $y$ gives
\an{\label{eq:ConvexCombination}
(1{-}m) f(x,y_1) + m &f(x,y_2) \geq f(x, (1{-}m)y_1 + m y_2),
}
for $m \in (0, 1)$, $y_1 \neq y_2$.  Using \eqref{eq:rescaler} with substitutions $\alpha = y_1/y$, $\alpha = y_2/y$, and $\alpha = [(1{-}m)y_1 + m y_2]/y$ in the terms in \eqref {eq:ConvexCombination}, where $y \in (0, \infty)$, yields:
\an{\label{eq:rescaler1}
(1&-m) \frac{y_1}{y} f\bigl(\frac{x y}{y_1},y \bigr) + m  \frac{y_2}{y} f \bigl(\frac{x y}{y_2}, y \bigr) \notag\\
& \geq \frac{(1{-}m)y_1 + m y_2}{y} f\bigl(\frac{x y}{(1{-}m)y_1 + m y_2}, y\bigr).
}
Let $x_1 \eqdef x y / y_1$ and $x_2 \eqdef x y / y_2$ represent the rescaled arguments for $f$ on the left side of \eqref {eq:rescaler1}.  We see that $x_1, x_2 \in (0, \infty)$ since $x \in (0, \infty)$ and $y,y_1,y_2 \neq 0$ have the same sign.

We try the ansatz that $x_1$ and $x_2$ can be combined convexly to yield the third rescaled argument on the right side of \eqref {eq:rescaler1}: 
\ab{
\frac{x y}{(1{-} m) y_1 {+} m y_2} = (1{-} h) x_1 {+} h x_2 {=} (1{-} h) \frac{x y}{y_1} {+} h \frac{x y}{y_2}. 
}
The ansatz has solution
\ab{
h &= \frac{m y_2}{(1{-}m) y_1 + m y_2},
\mbox{\ and }
1-h = \frac{(1{-}m) y_1}{(1{-}m) y_1 + m y_2}.
}
Note that $h \in (0,1)$ is assured because $y_1$ and $y_2$ have the same sign, $y_1 \neq y_2$, and $m \in (0,1)$.

Define $\phi \eqdef [{(1{-}m)y_1 + m y_2}]/{y}$.  Then $\phi > 0$ since  $y, y_1, y_2$ all have the same sign.  Substitution gives 
$(1{-}m){y_1}/{y} = (1{-}h) \phi$, and $m y_2/y = h \phi$, and \eqref{eq:rescaler1} becomes:
\ab{
(1{-}h) \phi  f(x_1,y) + h \phi  f(x_2, y)  \geq \phi f( (1{-}h) x_1 + h x_2, y).
}
After dividing both sides by $\phi > 0$, 
\an{
(1{-}h)  f(x_1,y) {+} h  f(x_2, y) \geq f( (1{-}h) x_1 + h x_2, y), \label{eq:Convexity}
}
which is convexity in $x$.  The case of strict convexity follows by substituting $>$ for $\geq$ throughout.
 \\
 
{\flushleft \ref{Lemma:Main:2}. \emph{Either
$f(x+d,1) < \ f(x,1) + d \, f(1,0)$ $\forall \ d > 0$, or 
$f(x+d,1) = \ f(x,1) + d \, f(1,0)$ $\forall \ d > 0$.} }

The strategy will be to show first that $f(x+d,1) \leq \ f(x,1) + d \, f(1,0)$.  Next, it is shown that if $f(x+d,1) < \ f(x,1) + d \, f(1,0)$ for any $d>0$, then it is true for all $d>0$, because convexity prevents $f(x+d,1)$ from ever returning to the line $f(x,1) + d \, f(1,0)$ for $d  > 0$.

By \eqref{eq:rescaler}, for $x, d > 0$, we have the equivalences
\an{
&f(x+d,1) \leq  \ f(x,1) + d \, f(1,0) \iff \label{eq:Monotone} \\
&(x+d) \, f \bigl(1, \frac{1}{x+d}\bigr) \leq  \ x \, f \bigl(1, \frac{1}{x} \bigr) + d \, f(1,0) \iff \notag \\
&f \bigl(1, \frac{1}{x+d} \bigr) \leq  \frac{x}{x+d} \, f \bigl(1, \frac{1}{x} \bigr) +   \frac{d}{x+d}\, f(1,0). \label{eq:renormed}
}
Since the $y$ arguments for $f(x,y)$ in \eqref {eq:renormed} are related by convex combination,
\ab{
\frac{1}{x+d} =  \frac{x}{x+d} \, \frac{1}{x} + \bigl(1- \frac{x}{x+d} \bigr) * 0  ,
}
then \eqref {eq:renormed} is just a statement of the convexity of $f(x,y)$ in $y$, as hypothesized.
The case of strict convexity follows by substituting $<$ for $\leq $, throughout.

Now, given $x >0$, suppose that for some $d_1 > 0$, 
\an{
f(x+d_1,1) <  \ f(x,1) + d_1 \, f(1,0). \label{eq:fx+d<}
}
We shall see that convexity then prevents $f(x+d,1)$ from ever returning to the line $f(x,1) + d \, f(1,0)$ for $d > 0$.  

We consider five points: $x < x+d_0 < x+d_1 < x+d_2 < x+d_3 < \infty$.
For readability, write $g(x) \equiv f(x,1)$ and $F \equiv f(1,0)$.  By convexity \eqref{eq:Convexity}, and hypothesis \eqref {eq:fx+d<},
\ab{
g(x&+d_0) 
\leq \bigl(1-\frac{d_0}{d_1} \bigr)   g(x)  + \frac{d_0}{d_1} g(x+d_1) \\
& < \bigl(1-\frac{d_0}{d_1} \bigr)   g(x)  + \frac{d_0}{d_1} (g(x) + d_1 F) 
= \ g(x) + d_0 \, F. 
}
and, by \eqref{eq:Convexity}, \eqref {eq:fx+d<}, and \eqref{eq:Monotone} (line $3$ below),
\ab{
g(x+ d_2)
&\leq  \frac{d_3 - d_2}{d_3 - d_1}   g(x +  d_1)  +  \frac{d_2 - d_1}{d_3 - d_1} g(x+ d_3) \\
& <  \frac{d_3 - d_2}{d_3 - d_1} (g(x) + d_1  F)  +  \frac{d_2 - d_1}{d_3 - d_1}g(x+ d_3) \\
& \leq  \frac{d_3 - d_2}{d_3 - d_1} (g(x) + d_1  F) 
+  \frac{d_2 - d_1}{d_3 - d_1} (g(x) + d_3 F ) \\
\iff & \ 
g(x+ d_2) (d_3 - d_1)& \\
& < (d_3 - d_2) \ ( g(x) + d_1  F)  + ( d_2 - d_1 ) (g(x) + d_3 F)\\
& = (d_3 - d_1) g(x) + d_2 (d_3 - d_1)  F \\
 \iff & \ 
 g(x+ d_2) 
< g(x) + d_2 \, F.
}

{\flushleft \ref{Lemma:Derivative}. \emph{When it exists, $\pf{}{x}f(x, 1) \leq  f(1,0)$.}}

Rearrangement of \eqref{eq:Monotone} gives
\ab{
\frac{f(x+d,1) {-}  f(x,1)}{d} \leq  f(1,0).
}
Hence when the limit exists,
\ab{
\lim_{d \goesto 0} \frac{f(x+d,1) - f(x,1)}{d} = \pf{f(x,1)}{x} \leq  f(1,0). \qedhere
}
\end{proof}

\begin{remark}\rm 
It would be clearly desirable to characterize the conditions for strict convexity in Kato's theorem, so that by Lemma \ref{Lemma:Main}, one would obtain strict convexity in Theorem \ref{Theorem:Main}, item \ref{item:Theorem:Main:1}, and strict monotonicity in items \ref{item:Theorem:Main:3} and \ref{item:Theorem:Main:4}.  Indeed, item \ref{Lemma:Main:2} is the best that can be offered in the way of strict inequality without strict convexity.  But the problem is more technical and is deferred to elsewhere.  

It is reasonable, nevertheless, to conjecture that the properties which produce strict convexity in the matrix case \citet[Theorem 4.1]{Friedland:1981} \cite[Theorem 1.1]{Nussbaum:1986:Convexity} extend to their Banach space versions:  i.e. for $a > 0$, when resolvent positive operator $A$ is irreducible \cite[p. 250]{Greiner:Voigt:and:Wolff:1981:Spectral} \citet[p. 41]{Arendt:and:Batty:1995}, then $\spb(\alpha A + \beta V)$ is strictly convex in $\beta$ if and only if $V$ is not a constant scalar.
\end{remark}

\subsection*{A Third Proof of Karlin's Theorem 5.2}
\ 

Karlin's proof was based on the Donsker-Varadhan variational formula for the spectral radius \citep{Donsker:and:Varadhan:1975}.  Kirkland et al. \cite{Kirkland:Li:and:Schreiber:2006} recently discovered another proof using entirely structural methods.  A third distinct proof of Karlin's theorem is seen here by application of Lemma \ref{Lemma:Main} to Cohen's theorem, combined with Friedland's equality condition \cite[Theorem 4.1]{Friedland:1981} (see also \cite{Nussbaum:1986:Convexity} for a different proof), as follows.

The expression in Karlin's Theorem 5.2 can be put in the form used in Theorem \ref{Theorem:Main}:
\ab{ 
\M(m)\D &= [(1{-}m)\I + m \P ] \D =  m(\P {-} \I) \D + \D \\
&= \alpha \A + \beta \D,
}
where $\A =  (\P {-} \I) \D$, $\alpha = m$, and $\beta = 1$.    

Since $\evt (\P {-} \I) \D = (\evt - \evt) \D = \0$, we see that $\spb(\A) = \spb( (\P {-} \I) \D ) = 0$.   Cohen's theorem gives that $\spb(\alpha \A + \beta \D)$ is convex in $\beta$, and thus by Lemma \ref{Lemma:Main}, $\spb(\alpha \A + \beta \D)$ is convex and non-increasing in $\alpha$.  Application of Lemma \ref{Lemma:Main} therefore yields that $\rho(\M(m)\D )$ is convex and non-increasing in $m$, and $\dfinline{\rho(\M(m)\D) }{m} \leq 0$, the derivative existing for all $m > 0$ when $\M(m)\D$ is irreducible.

From Friedland \cite[Theorem 4.1]{Friedland:1981}, strict convexity in $\beta$ occurs if $\P$ is irreducible and $\D \neq c \I$, for any $c > 0$.  By Lemma \ref{Lemma:Main} this implies $\spb(\M(m))$ is strictly decreasing and strictly convex in $m$.  \qedsym

\begin{remark}{\rm
The core of Kirkland et al.'s proof is their Lemma 4.1, which can be stated as 
\ab{
\evt \A \, (\uv(\A) \circ \vv(\A)) \geq \uv(\A)\tr \A \, \vv(\A) = \spb(\A),
}
with equality only when $\evt \A = \spb(\A) \evt$, where $\uv(\A)\tr$ and $\vv(\A)$ are the left and right eigenvectors of $\A$ associated with the Perron root $\spb(\A)$, and $\uv \circ \vv$ is the componentwise (Schur\revision{-}Hadamard) product.   Without the equality condition, their result is a special case of \citet[Theorem 3.2.5]{Bapat:and:Raghavan:1997}, but to obtain the equality condition requires an approach their novel proof provides.
}
\end{remark}

\begin{remark}\rm Schreiber and Lloyd-Smith \cite[Appendix B, Lemma 1]{Schreiber:and:Lloyd-Smith:2009} followed the reverse path and extended Kirkland et al's result on $\spb(\M(m)\D)$ to the form $\spb(\alpha \A {+} \D)$, where $\A$ is essentially nonnegative and $\D$ any diagonal matrix.
\end{remark}
\ 

Lemma \ref{Lemma:Main} can also be used as a new proof of an inequality of Lindqvist,  the special case considered in \citep[Theorem 2, pp. 260--261]{Lindqvist:2002}.  
\ \\
\begin {Theorem}[Lindqvist \protect{\citep[Theorem 2, subcase]{Lindqvist:2002}}]
\ 

Let $\A$ be an \revision{irreducible} $n \times n$ real matrix such that 1) $A_{ij} \geq 0$ for $i \neq j$, and 2) The left and right eigenvectors of $\A$, $\uv(\A)\tr$ and $\vv(\A)$, associated with eigenvalue $\spb(\A)$, satisfy $\uv(\A)\tr \vv(\A) = 1$.  Let $\D$ be an  $n \times n$ real diagonal matrix.  
Then
\an{\revision{\label{eq:Lindqvist} }
\spb(\A {+} \D) -  \spb(\A)  \geq  \uv(\A)\tr \, \D \, \vv(\A).
}
\end{Theorem}
\begin{proof}
Since $\A$ \revision{is an essentially nonnegative matrix, $\spb(\A)$ is an eigenvalue of multiplicity 1.  Consider the representation $\A = m \B - \D$, where $\B$ is essentially nonnegative and $m >0$.  Write $\spb \equiv \spb(\A)$.  As $\A$ is irreducible, $\uv \equiv \uv(\A)$, $\vv \equiv \vv(\A)$, with $\uvt \vv = 1$, $\evt \vv =1$, are unique, and the derivatives exist \cite{Deutsch:and:Neumann:1985} in the following} \cite[Sec. 9.1.1]{Caswell:2000}:
\ab{
\uv\tr \pf{( \A \vv)}{m} &= \uv\tr \left(\pf{\A} {m} \vv +  \A \pf{\vv} {m}\right) \\
&= \uv\tr \B \vv +  \spb \,   \uv\tr \pf{\vv}{m} 
\\
=\uv\tr  \pf{}{m}( \spb \, \vv ) 
& =  \uv\tr \left( \pf{\spb}{m} \, \vv + \spb \,   \pf{\vv}{m} \right) \\
&
= \pf{\spb}{m} + \spb \, \uv\tr   \pf{\vv}{m}.
}
Cancellation of terms $s \, \uv\tr \pfinline{\vv}{m}$ gives
\ab{
\pf{ \spb(\A)}{m} &= \uv(\A)\tr \pf{\A}{m} \vv (\A)
= \uv(\A)\tr \B \vv (\A) \leq \spb(\B),
}
the inequality coming from Lemma \ref{Lemma:Main}.  Scaling by $m$, subtracting $\D$, and substituting $m \B = \A + \D$, we get
\ab{
\uv(\A)\tr(m \B - \D) \vv (\A) = \spb(\A) &\leq \spb(m \B) - \uv(\A)\tr \D \vv(\A)\\
\iff
\uv(\A)\tr \D \vv(\A) & \leq \spb(\A {+} \D) - \spb(\A). \tag*{ \qedsym}
}
\end{proof}

\newcommand{\acknow}{I thank Prof. Shmuel Friedland for introducing me to the papers of Cohen, and for inviting me to speak on the early state of this work at the 16th International Linear Algebra Society Meeting in Pisa; Prof. Mustapha Mokhtar-Kharroubi for pointing out an error in a definition in the first version;  Laura Marie Herrmann for assistance with the literature search; and Arendt and Batty \citet{Arendt:and:Batty:1995} for guiding me to Kato \citet{Kato:1982:Superconvexity}.
}

\section*{Acknowledgements}
{\small \acknow}
{\footnotesize

}

\end{document}